\documentclass[10pt]{amsart}
\usepackage{amsmath, amsfonts, amssymb, amstext, cancel, amscd, graphicx, supertabular}
\usepackage[all,2cell,ps]{xy}

\newtheorem{palindromes}{Theorem}
\newtheorem{range}[palindromes]{Lemma}
\newtheorem{digits}[palindromes]{Lemma}

\begin{document}

\title[Palindromes in Different Bases]{Palindromes in Different Bases: \\ A Conjecture of J. Ernest Wilkins}

\author{Edray Herber Goins}
\address{Department of Mathematics \\ Purdue University \\ 150 North University Street \\ West Lafayette, IN 47907}
\email{egoins@math.purdue.edu}

\begin{abstract} We show that there exist exactly 203 positive integers $N$ such that for some integer $d \geq 2$ this number is a $d$-digit palindrome base 10 as well as a $d$-digit palindrome for some base $b$ different from 10.  To be more precise, such $N$ range from 22 to 9986831781362631871386899. \end{abstract}
\dedicatory{This paper is dedicated to J. Ernest Wilkins.}

\maketitle

\section{Introduction}

During the summer of 2004, while meeting at the Conference for African-American Researchers in the Mathematical Sciences (CAARMS), the author had a short conversation with J. Ernest Wilkins.  He was interested in palindromes which remain palindromes when expressed in a different base system.  For example, 207702 is a 6-digit palindrome expressed in base 10 (written as $[2,0,7,7,0,2]_{10}$) as well as 6-digit palindrome expressed in base 8 (written as $[6,2,5,5,2,6]_8$).  He posed the following question:

\begin{quote} \textit{Does there exist a positive integer $N$ which is an 8-digit palindrome base 10 as well as an 8-digit palindrome for some base $b$ different from 10?}  \end{quote}

\noindent He suspected that the answer would be no, but, without being well-versed in the art of computer programming, could not find a definitive proof using pen and paper.

It is natural to generalize this question to any number of digits.  The main result of this exposition is as follows:

\begin{palindromes} \label{palindromes} There exist exactly 203 positive integers $N$ such that for some integer $d \geq 2$ this number is a $d$-digit palindrome base 10 as well as a $d$-digit palindrome for some base $b$ different from 10.  To be more precise, such $N$ range from 22 to 9986831781362631871386899, and
\[ d = 2, \, 3, \, 4, \, 5, \, 6, \, 7, \, 9, \, 11, \, 13, \, 15, \, 17, \, 19, \, 21, \, 23, \, 25. \] \end{palindromes}

\noindent (We assume that $d \geq 1$ because any positive integer $N < 10$ is trivially a 1-digit palindrome base $b$ for any $b > N$.) A complete list of these palindromes can be found in the appendix.

The author found this result using a few theoretically trivial yet computationally frustrating inequalities, then parallelized the search using several high powered machines at Purdue University.  Almost all of the computations were done using \texttt{gridMathematica}.  Wilkins conjectured that the only \emph{even} $d$ are $d = 2, \, 4, \, 6$; this result is positive verification of his conjecture.   This article is a bit different from \cite{MR2412782}: Those authors consider positive integers $N = 10^n \pm 1$ which are palindromes base 10 as well as palindromes base 2 -- but the number of digits is not fixed among the bases.

\section{Computational Set-Up}

Fix an integer $b \geq 2$.  For each positive integer $N$, the expression $[ a_d, \dots, \, a_2, \, a_1 ]_b$ will denote the unique base $b$ expansion
\[ N = a_d \, b^{d-1} + \cdots + a_2 \, b + a_1 \qquad \text{where} \qquad 0 \leq a_k < b. \]
\noindent We will call this the \emph{base $b$ representation of $N$}.  Moreover, we will say that $N$ is a \emph{$d$-digit number base $b$} if $a_d$ is nonzero; and that such a $d$-digit number is a \emph{$d$-digit palindrome base $b$} if $a_{d-k+1} = a_k$ for $k = 1, \, 2, \dots, \, d$.  For example, $N = 207702$ may be expressed as $[2,0,7,7,0,2]_{10}$ for $b = 10$, or $[6,2,5,5,2,6]_8$ for $b= 8$.  In particular, $N$ is a 6-digit palindrome base 10 as well as a 6-digit palindrome base 8.

There are only finitely many $d$-digit numbers base 10 which are also $d$-digit numbers for some other base $b$:

\begin{range} \label{range} If $N$ is a positive $d$-digit number base 10 which is also $d$-digit number for some base $b \neq 10$ then $d \leq 26$. \end{range}

\begin{proof} (Wilkins himself suggested this proof.) Upon fixing such an integer $N$, the base $b$ must satisfy
\[ \left \lfloor \frac {\log N}{\log 10} \right \rfloor + 1 = d = \left \lfloor \frac {\log N}{\log b} \right \rfloor + 1 \]
\noindent in terms of the greatest integer function $\lfloor \cdot \rfloor$.  In fact, given any real number $x$ we have the inequality $x \leq \lfloor x \rfloor + 1 \leq x + 1$, so it is easy to see that
\[ -1 \leq \frac {\log N}{\log 10} - \frac {\log N}{\log b} \leq 1 \qquad \implies \qquad 10^{1/(1 + \log 10/\log N)} \leq b \leq 10^{1/(1 - \log 10/\log N)}. \]
\noindent When $N \geq 10^{26}$, i.e., $d > 26$, this forces $b = 10$. \end{proof}

The following result gives computable ranges for $N$:

\begin{digits} \label{digits} Say that $N$ is a positive integer $N$ which is both a $d$-digit palindrome base 10 as well as a $d$-digit palindrome for some base $b \neq 10$.  Then either $d \leq 14$, or else $d$, $b$, and $N$ are related as in the following table:
\vskip .2in
\tablehead{\hline Digits $d$ & Base $b$ & Range for $N$ \\ \hline \hline}
\tabletail{\hline}
\begin{center} \begin{supertabular}{|c|c|c|}
15 & 9 & $10^{14} < N < 9^{15}$ \\ & 11 & $11^{14} < N < 10^{15}$ \\ \hline
16 & 9 & $10^{15} < N < 9^{16}$ \\ \hline
17 & 9 & $10^{16} < N < 9^{17}$ \\ & 11 & $11^{16} < N < 10^{17}$ \\ \hline
18 & 9 & $10^{17} < N < 9^{18}$ \\ \hline
19 & 9 & $10^{18} < N < 9^{19}$ \\ & 11 & $11^{18} < N \leq 10^{19}$ \\ \hline
20 & 9 & $10^{19} < N < 9^{20}$ \\ \hline
21 & 9 & $10^{20} < N < 9^{21}$ \\ & 11 & $11^{20} < N \leq 10^{21}$ \\ \hline
23 & 11 & $11^{22} < N < 10^{23}$ \\ \hline
25 & 11 & $11^{24} < N < 10^{25}$ \\
\end{supertabular} \end{center}
\end{digits}

\begin{proof} According to Lemma \ref{range}, it suffices to consider those integers $N$ satisfying the double inequality $10^1 < N < 10^{26}$.  Recall that
\[ d = \left \lfloor \frac {\log N}{\log b} \right \rfloor + 1. \]
\noindent When $15 \leq d \leq 21$, this forces $9 \leq b \leq 11$; and when $22 \leq d \leq 26$, this forces $10 \leq b \leq 11$.     It suffices then to show that $d \neq 22, \, 24, \, 26$.  Following an observation of Wilkins, we see that no $d$-digit palindrome base 10 can also be a $d$-digit palindrome base 11 when $d$ is even:  Indeed, write $[ c_d, \dots, \, c_2, \, c_1 ]_{10}$ and $[ a_d, \dots, \, a_2, \, a_1 ]_{11}$ as the base 10 and base 11 representations of $N$, respectively, where the leading coefficient satisfies $0 < a_d < 11$.  Then we find
\begin{equation} \begin{aligned} a_d & = a_1 \equiv a_d \cdot 11^{d-1} + a_{d-1} \cdot 11^{d-2} + \cdots + a_2 \cdot 11 + a_1 \pmod{11} \\ & = N = c_d \cdot 10^{d-1} + c_{d-1} \cdot 10^{d-2} + \cdots + c_2 \cdot 10 + c_1 \\ & \equiv 0 + \cdots + (c_{d-1} - c_2) - (c_d - c_1)  \pmod{11} \\ & \equiv 0 \pmod{11} \end{aligned} \end{equation}
\noindent which is a contradiction. \end{proof}

\section{Implementation}

Here is the actual \texttt{Mathematica} code.  Given a pair of $d$-digit integers $\{ N_1, N_2 \}$ and a base $b$, the output is a list of $d$-digit palindromes base $b$ in the range $N_1 \leq N < N_2$ which are palindromes for some base different from $b$.  In practice, we set $N_1 = 10$, $N_2 = 10^{26}$, and $b = 10$ -- although for large enough $N$ it seems computationally more efficient to set $b=11$.  The built-in \textsf{Mathematica} command \texttt{RealDigits[N,b]} returns $\{ \{ a_d, \dots, \, a_1 \}, d\}$, as related to the base $b$ expansion $N = a_d \, b^{d-1} + \cdots + a_2 \, b + a_1$; while \texttt{FromDigits[list, b]} undoes this command and returns $N$.

\begin{verbatim} PalindromeSearch[{N1_Integer, N2_Integer, b_Integer}] := 
   Module[{d, FoundList, TestPalindrome, BaseList},
   
      d = RealDigits[N1,b][[2]]; (* number of digits *)
      FoundList = {}; (* list of found palindromes *)
      
      For[ 
         n = Floor[ N1/b^Floor[d/2] ], 
         n < Floor[ N2/b^Floor[d/2] ], 
         n++,
         (* n denotes the first d/2 digits of the palindrome *)
         
         TestPalindrome = FromDigits[Join[
            RealDigits[n,b][[1]], 
            Take[ Reverse[RealDigits[n,b][[1]]], -Floor[d/2] ]
         ],b];
         (* reconstructs the d-digit palindrome from n *)
         
         BaseList = Select[
            Range[
               Ceiling[ b^(1/(1+Log[TestPalindrome,b])) ],
               Floor[ b^(1/(1-Log[TestPalindrome,b])) ]
            ],
            RealDigits[TestPalindrome,#][[1]] == 
            Reverse[ RealDigits[TestPalindrome,#][[1]] ]
            &&
            RealDigits[TestPalindrome,#][[2]] == d &
         ];
         (* a list of bases for which TestPalindrome is also a
         d-digit palindrome *)
         
         If[
            Length[BaseList] > 1,
            AppendTo[ FoundList, {TestPalindrome, BaseList} ];
            (* if base is not b, add to list *)
         ];
                  
      ];
  
      Return[FoundList]; (* return complete list *)
   ]
\end{verbatim}

\section{Enumeration of Palindromes}

Lemma \ref{digits} gives effective computing ranges.    The first $d \leq 14$ digits took about a day.    These last $15 \leq d < 26$ took about fifteen months running on twenty processors each -- for a total of about twelve years computing time!  The results form the basis of Theorem \ref{palindromes} and the table below.  Recall that $[ a_d, \dots, \, a_2, \, a_1 ]_b$ denotes the expansion $N = a_d \, b^{d-1} + \cdots + a_2 \, b + a_1$.

\vskip .2in
\begin{center}
\scriptsize
\tablefirsthead{\hline Digits $d$ & Integer $N$ & Base $b$ Representations \\ \hline \hline}
\tablehead{\hline \multicolumn{3}{|l|}{\small\sl continued from previous page} \\ \hline Digits $d$ & Integer $N$ & Base $b$ Representations \\ \hline \hline}
\tabletail{\hline \hline \multicolumn{3}{|r|}{\small\sl continued on the next page} \\ \hline}
\tablelasttail{\hline}
\begin{supertabular}{|c|c|c|}
2 & 22 & $[2,2]_{10}$, $[1,1]_{21}$ \\ & 33 & $[3,3]_{10}$, $[1,1]_{32}$ \\ & 44 & $[4,4]_{10}$, $[2,2]_{21}$, $[1,1]_{43}$ \\ & 55 & $[5,5]_{10}$, $[1,1]_{54}$ \\ & 66 & $[6,6]_{10}$, $[3,3]_{21}$, $[2,2]_{32}$, $[1,1]_{65}$ \\ & 77 & $[7,7]_{10}$, $[1,1]_{76}$ \\ & 88 & $[8,8]_{10}$, $[4,4]_{21}$, $[2,2]_{43}$, $[1,1]_{87}$ \\ & 99 & $[9,9]_{10}$, $[3,3]_{32}$, $[1,1]_{98}$ \\ \hline
3 & 111& $[3,0,3]_6$, $[1,1,1]_{10}$ \\ & 121 & $[2,3,2]_7$, $[1,7,1]_8$, $[1,2,1]_{10}$ \\ & 141 & $[3,5,3]_6$, $[1,4,1]_{10}$ \\ & 171 & $[3,3,3]_7$, $[1,7,1]_{10}$ \\ & 181 & $[1,8,1]_{10}$, $[1,3,1]_{12}$\\ & 191 & $[5,1,5]_6$, $[2,3,2]_9$, $[1,9,1]_{10}$ \\ & 222 & $[2,2,2]_{10}$, $[1,4,1]_{13}$ \\ & 232 & $[2,3,2]_{10}$, $[1,10,1]_{11}$ \\ & 242 & $[4,6,4]_7$, $[2,4,2]_{10}$ \\ & 282 & $[3,4,3]_9$, $[2,8,2]_{10}$ \\ & 292 & $[5,6,5]_7$, $[4,4,4]_8$, $[2,9,2]_{10}$ \\ & 313 & $[3,1,3]_{10}$, $[1,11,1]_{13}$ \\ & 323 & $[3,2,3]_{10}$, $[1,9,1]_{14}$ \\ & 333 & $[5,1,5]_8$, $[3,3,3]_{10}$ \\ & 343 & $[3,4,3]_{10}$, $[2,9,2]_{11}$, $[1,1,1]_{18}$ \\ & 353 & $[3,5,3]_{10}$, $[2,1,2]_{13}$, $[1,6,1]_{16}$ \\ & 373 & $[5,6,5]_8$, $[4,5,4]_9$, $[3,7,3]_{10}$ \\ & 414 & $[6,3,6]_8$, $[4,1,4]_{10}$ \\ & 444 & $[4,4,4]_{10}$, $[2,8,2]_{13}$ \\ & 454 & $[4,5,4]_{10}$, $[3,8,3]_{11}$ \\ & 464 & $[5,6,5]_9$, $[4,6,4]_{10}$, $[2,5,2]_{14}$ \\ & 484 & $[4,8,4]_{10}$, $[1,2,1]_{21}$ \\ & 494 & $[4,9,4]_{10}$, $[1,12,1]_{17}$ \\ & 505 & $[5,0,5]_{10}$, $[1,10,1]_{18}$, $[1,3,1]_{21}$ \\ & 545 & $[5,4,5]_{10}$, $[1,15,1]_{17}$ \\ & 555 & $[6,7,6]_9$, $[5,5,5]_{10}$, $[3,10,3]_{12}$ \\ & 565 & $[5,6,5]_{10}$, $[4,7,4]_{11}$ \\ & 575 & $[5,7,5]_{10}$, $[3,5,3]_{13}$ \\ & 595 & $[5,9,5]_{10}$, $[1,15,1]_{18}$, $[1,5,1]_{22}$ \\ & 616 & $[6,1,6]_{10}$, $[4,3,4]_{12}$ \\ & 626 & $[6,2,6]_{10}$, $[2,7,2]_{16}$, $[1,0,1]_{25}$ \\ & 646 & $[7,8,7]_{9}$, $[6,4,6]_{10}$ \\ & 656 & $[8,0,8]_9$, $[6,5,6]_{10}$ \\ & 666 & $[6,6,6]_{10}$, $[3,12,3]_{13}$, $[1,16,1]_{19}$ \\ & 676 & $[6,7,6]_{10}$, $[5,6,5]_{11}$, $[4,8,4]_{12}$, $[1,2,1]_{25}$ \\ & 686 & $[6,8,6]_{10}$, $[2,2,2]_{18}$ \\ & 717 & $[7,1,7]_{10}$, $[3,9,3]_{14}$ \\ & 727 & $[7,2,7]_{10}$, $[1,11,1]_{22}$ \\ & 737 & $[7,3,7]_{10}$, $[5,1,5]_{12}$, $[1,9,1]_{23}$ \\ & 757 & $[7,5,7]_{10}$, $[1,15,1]_{21}$, $[1,1,1]_{27}$ \\ & 767 & $[7,6,7]_{10}$, $[2,11,2]_{17}$ \\ & 787 & $[7,8,7]_{10}$, $[6,5,6]_{11}$, $[3,1,3]_{16}$ \\ & 797 & $[7,9,7]_{10}$, $[5,6,5]_{12}$, $[4,9,4]_{13}$ \\ & 818 & $[8,1,8]_{10}$, $[2,14,2]_{17}$ \\ & 828 & $[8,2,8]_{10}$, $[3,10,3]_{15}$ \\ & 838 & $[8,3,8]_{10}$, $[2,6,2]_{19}$, $[1,4,1]_{27}$ \\ & 848 & $[8,4,8]_{10}$, $[2,11,2]_{18}$ \\ & 858 & $[8,5,8]_{10}$, $[4,5,4]_{14}$, $[3,12,3]_{15}$ \\ & 888 & $[8,8,8]_{10}$, $[3,14,3]_{15}$ \\ & 898 & $[8,9,8]_{10}$, $[7,4,7]_{11}$, $[1,16,1]_{23}$ \\ & 909 & $[9,0,9]_{10}$, $[7,5,7]_{11}$ \\ & 919 & $[9,1,9]_{10}$, $[4,1,4]_{15}$, $[1,7,1]_{27}$ \\ & 929 & $[9,2,9]_{10}$, $[1,3,1]_{29}$ \\ & 949 & $[9,4,9]_{10}$, $[4,3,4]_{15}$\\ & 979 & $[9,7,9]_{10}$, $[4,5,4]_{15}$, $[3,13,3]_{16}$ \\ & 989 & $[9,8,9]_{10}$, $[3,7,3]_{17}$, $[2,5,2]_{21}$, $[1,12,1]_{26}$ \\ & 999 & $[9,9,9]_{10}$, $[5,1,5]_{14}$ \\ \hline
4 & 3663 & $[7,1,1,7]_8$, $[3,6,6,3]_{10}$ \\ & 6776 & $[6,7,7,6]_{10}$, $[3,1,1,3]_{13}$ \\ & 8008 & $[8,0,0,8]_{10}$, $[4,7,7,4]_{12}$ \\ & 8778 & $[8,7,7,8]_{10}$, $[3,12,12,3]_{13}$ \\ \hline
5 & 13131 & $[3,1,5,1,3]_8$, $[1,3,1,3,1]_{10}$ \\ & 13331 & $[3,2,0,2,3]_8$, $[1,3,3,3,1]_{10}$ \\ & 16561 & $[6,6,1,6,6]_7$, $[1,6,5,6,1]_{10}$ \\ & 25752 & $[3,8,2,8,3]_9$, $[2,5,7,5,2]_{10}$ \\ & 26462 & $[6,3,5,3,6]_8$, $[2,6,4,6,2]_{10}$ \\ & 26662 & $[6,4,0,4,6]_8$, $[2,6,6,6,2]_{10}$ \\ & 26962 & $[2,6,9,6,2]_{10}$, $[1,9,2,9,1]_{11}$ \\ & 27472 & $[4,1,6,1,4]_9$, $[2,7,4,7,2]_{10}$ \\ & 30103 & $[7,2,6,2,7]_8$, $[3,0,1,0,3]_{10}$ \\ & 30303 & $[7,3,1,3,7]_8$, $[3,0,3,0,3]_{10}$ \\ & 35953 & $[3,5,9,5,3]_{10}$, $[1,8,9,8,1]_{12}$ \\ & 38183 & $[3,8,1,8,3]_{10}$, $[1,8,9,8,1]_{11}$ \\ & 39593 & $[3,9,5,9,3]_{10}$, $[1,0,6,0,1]_{14}$ \\ & 40504 & $[4,0,5,0,4]_{10}$, $[2,8,4,8,2]_{11}$ \\ & 42324 & $[6,4,0,4,6]_9$, $[4,2,3,2,4]_{10}$ \\ & 43934 & $[4,3,9,3,4]_{10}$, $[2,1,5,1,2]_{12}$ \\ & 49294 & $[4,9,2,9,4]_{10}$, $[3,4,0,4,3]_{11}$ \\ & 50605 & $[7,6,3,6,7]_9$, $[5,0,6,0,5]_{10}$ \\ & 52825 & $[5,2,8,2,5]_{10}$, $[3,6,7,6,3]_{11}$ \\ & 56265 & $[5,6,2,6,5]_{10}$, $[1,12,7,12,1]_{13}$ \\ & 59095 & $[5,9,0,9,5]_{10}$, $[1,7,7,7,1]_{14}$ \\ & 60106 & $[6,0,1,0,6]_{10}$, $[1,2,12,2,1]_{15}$ \\ & 63936 & $[6,3,9,3,6]_{10}$, $[4,4,0,4,4]_{11}$ \\ & 67576 & $[6,7,5,7,6]_{10}$, $[1,5,0,5,1]_{15}$ \\ & 75157 & $[7,5,1,5,7]_{10}$, $[5,1,5,1,5]_{11}$ \\ & 88888 & $[8,8,8,8,8]_{10}$, $[4,3,5,3,4]_{12}$ \\ & 90209 & $[9,0,2,0,9]_{10}$, $[1,6,0,6,1]_{16}$ \\ & 94049 & $[9,4,0,4,9]_{10}$, $[1,6,15,6,1]_{16}$ \\ & 94249 & $[9,4,2,4,9]_{10}$, $[1,2,3,2,1]_{17}$ \\ & 96369 & $[9,6,3,6,9]_{10}$, $[1,7,8,7,1]_{16}$ \\ & 98689 & $[9,8,6,8,9]_{10}$, $[1,8,1,8,1]_{16}$ \\ \hline
6 & 207702 & $[6,2,5,5,2,6]_8$, $[2,0,7,7,0,2]_{10}$ \\ & 546645 & $[5,4,6,6,4,5]_{10}$, $[1,0,3,3,0,1]_{14}$ \\ & 646646 & $[6,4,6,6,4,6]_{10}$, $[2,7,2,2,7,2]_{12}$ \\ \hline
7 & 1496941 & $[5, 5, 5, 3, 5, 5, 5]_8$, $[1,4,9,6,9,4,1]_{10}$ \\ & 1540451 & $[2, 8, 0, 7, 0, 8, 2]_9$, $[1,5,4,0,4,5,1]_{10}$ \\ & 1713171 & $[3, 2, 0, 1, 0, 2, 3]_9$, $[1,7,1,3,1,7,1]_{10}$ \\ & 1721271 & $[3, 2, 1, 3, 1, 2, 3]_9$, $[1,7,2,1,2,7,1]_{10}$ \\ & 1828281 & $[3, 3, 8, 5, 8, 3, 3]_9$, $[1,8,2,8,2,8,1]_{10}$ \\ & 1877781 & $[3, 4, 7, 1, 7, 4, 3]_9$, $[1,8,7,7,7,8,1]_{10}$ \\ & 1885881 & $[3, 4, 8, 3, 8, 4, 3]_9$, $[1,8,8,5,8,8,1]_{10}$ \\ & 1935391 & $[7, 3, 0, 4, 0, 3, 7]_8$, $[1,9,3,5,3,9,1]_{10}$ \\ & 1970791 & $[7, 4, 1, 1, 1, 4, 7]_8$, $[1,9,7,0,7,9,1]_{10}$ \\ & 2401042 & $[4, 4, 5, 8, 5, 4, 4]_9$, $[2,4,0,1,0,4,2]_{10}$ \\ & 2434342 & $[4, 5, 2, 0, 2, 5, 4]_9$, $[2,4,3,4,3,4,2]_{10}$ \\ & 2442442 & $[4, 5, 3, 2, 3, 5, 4]_9$, $[2,4,4,2,4,4,2]_{10}$ \\ &  2450542 & $[4, 5, 4, 4, 4, 5, 4]_9$, $[2,4,5,0,5,4,2]_{10}$ \\ & 2956592 & $[2,9,5,6,5,9,2]_{10}$, $[1, 7, 3, 10, 3, 7, 1]_{11}$ \\ & 2968692 & $[2,9,6,8,6,9,2]_{10}$, $[1, 7, 4, 8, 4, 7, 1]_{11}$ \\ & 3106013 & $[5, 7, 5, 3, 5, 7, 5]_9$, $[3,1,0,6,0,1,3]_{10}$ \\ & 3114113 & $[5, 7, 6, 5, 6, 7, 5]_9$, $[3,1,1,4,1,1,3]_{10}$ \\ & 3122213 & $[5, 7, 7, 7, 7, 7, 5]_9$, $[3,1,2,2,2,1,3]_{10}$ \\ & 3163613 & $[5, 8, 5, 1, 5, 8, 5]_9$, $[3,1,6,3,6,1,3]_{10}$ \\ & 3171713 & $[5, 8, 6, 3, 6, 8, 5]_9$, $[3,1,7,1,7,1,3]_{10}$ \\ & 3192913 & $[3,1,9,2,9,1,3]_{10}$, $[1, 0, 9, 11, 9, 0, 1]_{12}$ \\ & 3262623 & $[3,2,6,2,6,2,3]_{10}$, $[1, 9, 2, 9, 2, 9, 1]_{11}$ \\ & 3274723 & $[3,2,7,4,7,2,3]_{10}$, $[1, 9, 3, 7, 3, 9, 1]_{11}$ \\ & 3286823 & $[3,2,8,6,8,2,3]_{10}$, $[1, 9, 4, 5, 4, 9, 1]_{11}$ \\ & 3298923 & $[3,2,9,8,9,2,3]_{10}$, $[1, 9, 5, 3, 5, 9, 1]_{11}$ \\ & 3303033 & $[6, 1, 8, 3, 8, 1, 6]_9$, $[3,3,0,3,0,3,3]_{10}$ \\ & 3360633 & $[6, 2, 8, 1, 8, 2, 6]_9$, $[3,3,6,0,6,3,3]_{10}$, $[1, 9, 9, 5, 9, 9, 1]_{11}$ \\ & 3372733 & $[3,3,7,2,7,3,3]_{10}$, $[1, 9, 10, 3, 10, 9, 1]_{11}$ \\ & 4348434 & $[4,3,4,8,4,3,4]_{10}$, $[2, 5, 0, 0, 0, 5, 2]_{11}$ \\ & 4410144 & $[4,4,1,0,1,4,4]_{10}$, $[2, 5, 4, 2, 4, 5, 2]_{11}$ \\ & 4422244 & $[4,4,2,2,2,4,4]_{10}$, $[2, 5, 5, 0, 5, 5, 2]_{11}$ \\ & 4581854 & $[4,5,8,1,8,5,4]_{10}$, $[2, 6, 4, 10, 4, 6, 2]_{11}$ \\ & 4593954 & $[4,5,9,3,9,5,4]_{10}$, $[2, 6, 5, 8, 5, 6, 2]_{11}$ \\ & 5641465 & $[5,6,4,1,4,6,5]_{10}$, $[1, 10, 8, 0, 8, 10, 1]_{12}$ \\ & 5643465 & $[5,6,4,3,4,6,5]_{10}$, $[3, 2, 0, 5, 0, 2, 3]_{11}$ \\ & 5655565 & $[5,6,5,5,5,6,5]_{10}$, $[3, 2, 1, 3, 1, 2, 3]_{11}$ \\ & 5667665 & $[5,6,6,7,6,6,5]_{10}$, $[3, 2, 2, 1, 2, 2, 3]_{11}$ \\ & 5741475 & $[5,7,4,1,4,7,5]_{10}$, $[3, 2, 7, 1, 7, 2, 3]_{11}$ \\ & 7280827 & $[7,2,8,0,8,2,7]_{10}$, $[4, 1, 2, 3, 2, 1, 4]_{11}$ \\ & 7292927 & $[7,2,9,2,9,2,7]_{10}$, $[4, 1, 3, 1, 3, 1, 4]_{11}$ \\ & 8364638 & $[8,3,6,4,6,3,8]_{10}$, $[2, 9, 7, 4, 7, 9, 2]_{12}$ \\ & 8710178 & $[8,7,1,0,1,7,8]_{10}$, $[4, 10, 0, 10, 0, 10, 4]_{11}$ \\ & 8722278 & $[8,7,2,2,2,7,8]_{10}$, $[4, 10, 1, 8, 1, 10, 4]_{11}$ \\ & 8734378 & $[8,7,3,4,3,7,8]_{10}$, $[4, 10, 2, 6, 2, 10, 4]_{11}$ \\ & 8746478 & $[8,7,4,6,4,7,8]_{10}$, $[4, 10, 3, 4, 3, 10, 4]_{11}$ \\ & 8758578 & $[8,7,5,8,5,7,8]_{10}$, $[4, 10, 4, 2, 4, 10, 4]_{11}$ \\ & 8820288 & $[8,8,2,0,2,8,8]_{10}$, $[4, 10, 8, 4, 8, 10, 4]_{11}$ \\ & 8832388 & $[8,8,3,2,3,8,8]_{10}$, $[4, 10, 9, 2, 9, 10, 4]_{11}$ \\ & 8844488 & $[8,8,4,4,4,8,8]_{10}$, $[4, 10, 10, 0, 10, 10, 4]_{11}$ \\ & 8864688 & $[8,8,6,4,6,8,8]_{10}$, $[1, 10, 11, 4, 11, 10, 1]_{13}$ \\ & 9046409 & $[9,0,4,6,4,0,9]_{10}$, $[1, 2, 11, 6, 11, 2, 1]_{14}$ \\ & 9578759 & $[9,5,7,8,7,5,9]_{10}$, $[1, 3, 11, 4, 11, 3, 1]_{14}$ \\ & 9813189 & $[9,8,1,3,1,8,9]_{10}$, $[1, 4, 3, 6, 3, 4, 1]_{14}$ \\ & 9963699 & $[9,9,6,3,6,9,9]_{10}$, $[3, 4, 0, 6, 0, 4, 3]_{12}$ \\ \hline
9 & 130535031 & $[7, 6, 1, 7, 4, 7, 1, 6, 7]_8$, $[1,3,0,5,3,5,0,3,1]_{10}$ \\ & 167191761 & $[3, 7, 8, 5, 3, 5, 8, 7, 3]_9$, $[1,6,7,1,9,1,7,6,1]_{10}$ \\ & 181434181 & $[4, 1, 8, 3, 5, 3, 8, 1, 4]_9$, $[1,8,1,4,3,4,1,8,1]_{10}$ \\ & 232000232 & $[5, 3, 4, 4, 8, 4, 4, 3, 5]_9$, $[2,3,2,0,0,0,2,3,2]_{10}$ \\ & 356777653 & $[3,5,6,7,7,7,6,5,3]_{10}$, $[1, 7, 3, 4, 3, 4, 3, 7, 1]_{11}$ \\ & 362151263 & $[3,6,2,1,5,1,2,6,3]_{10}$, $[1, 7, 6, 4, 7, 4, 6, 7, 1]_{11}$ \\ & 382000283 & $[8, 7, 7, 7, 1, 7, 7, 7, 8]_9$, $[3,8,2,0,0,0,2,8,3]_{10}$ \\ & 489525984 & $[4,8,9,5,2,5,9,8,4]_{10}$, $[2, 3, 1, 3, 6, 3, 1, 3, 2]_{11}$ \\ & 492080294 & $[4,9,2,0,8,0,2,9,4]_{10}$, $[2, 3, 2, 8, 4, 8, 2, 3, 2]_{11}$ \\ & 520020025 & $[5,2,0,0,2,0,0,2,5]_{10}$, $[1, 2, 6, 1, 10, 1, 6, 2, 1]_{12}$ \\ & 537181735 & $[5,3,7,1,8,1,7,3,5]_{10}$, $[2, 5, 6, 2, 5, 2, 6, 5, 2]_{11}$ \\ & 713171317 & $[7,1,3,1,7,1,3,1,7]_{10}$, $[1, 7, 10, 10, 0, 10, 10, 7, 1]_{12}$ \\ & 796212697 & $[7,9,6,2,1,2,6,9,7]_{10}$, $[1, 10, 2, 7, 9, 7, 2, 10, 1]_{12}$ \\ & 952404259 & $[9,5,2,4,0,4,2,5,9]_{10}$, $[1, 2, 2, 4, 1, 4, 2, 2, 1]_{13}$ \\ & 998111899 & $[9,9,8,1,1,1,8,9,9]_{10}$, $[4, 7, 2, 4, 5, 4, 2, 7, 4]_{11}$ \\ & 999454999 & $[9,9,9,4,5,4,9,9,9]_{10}$, $[4, 7, 3, 1, 9, 1, 3, 7, 4]_{11}$ \\ \hline
11 & 39276067293 & $[3,9,2,7,6,0,6,7,2,9,3]_{10}$, \\ & & $[1, 5, 7, 2, 5, 3, 5, 2, 7, 5, 1]_{11}$ \\ & 39453235493 & $[3,9,4,5,3,2,3,5,4,9,3]_{10}$, \\ & & $[1, 5, 8, 0, 6, 3, 6, 0, 8, 5, 1]_{11}$ \\ & 42521012524 & $[4,2,5,2,1,0,1,2,5,2,4]_{10}$, \\ & & $[1, 7, 0, 4, 0, 0, 0, 4, 0, 7, 1]_{11}$ \\ & 73183838137 & $[7,3,1,8,3,8,3,8,1,3,7]_{10}$, \\ & & $[1, 2, 2, 2, 5, 1, 5, 2, 2, 2, 1]_{12}$ \\ \hline
13 & 1400232320041 & $[4, 8, 5, 5, 2, 1, 7, 1, 2, 5, 5, 8, 4]_9$, \\ & & $[1,4,0,0,2,3,2,3,2,0,0,4,1]_{10}$ \\ & 2005542455002 & $[7, 0, 8, 1, 5, 8, 0, 8, 5, 1, 8, 0, 7]_9$, \\ & & $[2,0,0,5,5,4,2,4,5,5,0,0,2]_{10}$ \\ & 2024099904202 & $[7, 1, 4, 4, 5, 0, 0, 0, 5, 4, 4, 1, 7]_9$, \\ & & $[2,0,2,4,0,9,9,9,0,4,2,0,2]_{10}$ \\ & 2081985891802 & $[7, 3, 3, 0, 8, 6, 4, 6, 8, 0, 3, 3, 7]_9$, \\ & & $[2,0,8,1,9,8,5,8,9,1,8,0,2]_{10}$ \\ & 4798641468974 & $[4,7,9,8,6,4,1,4,6,8,9,7,4]_{10}$, \\ & & $[1, 5, 9, 0, 1, 0, 2, 0, 1, 0, 9, 5, 1]_{11}$ \\ \hline
15 & 101904010409101 & $[4, 4, 0, 7, 2, 7, 0, 5, 0, 7, 2, 7, 0, 4, 4]_9$, \\ & & $[1,0,1,9,0,4,0,1,0,4,0,9,1,0,1]_{10}$ \\ & 149285434582941 & $[6, 4, 6, 5, 1, 5, 7, 1, 7, 5, 1, 5, 6, 4, 6]_9$, \\ & & $[1,4,9,2,8,5,4,3,4,5,8,2,9,4,1]_{10}$ \\ & 149819212918941 & $[6, 4, 8, 4, 1, 6, 5, 1, 5, 6, 1, 4, 8, 4, 6]_9$, \\ & & $[1,4,9,8,1,9,2,1,2,9,1,8,9,4,1]_{10}$ \\ & 463906656609364 & $[4,6,3,9,0,6,6,5,6,6,0,9,3,6,4]_{10}$, \\ & & $[1, 2, 4, 8, 10, 6, 7, 8, 7, 6, 10, 8, 4, 2, 1]_{11}$ \\ \hline
17 & 11111059395011111 & $[5, 8, 8, 6, 1, 8, 8, 6, 3, 6, 8, 8, 1, 6, 8, 8, 5]_9$, \\ & & $[1,1,1,1,1,0,5,9,3,9,5,0,1,1,1,1,1]_{10}$ \\ & 11199701210799111 & $[6, 0, 3, 5, 0, 7, 5, 8, 3, 8, 5, 7, 0, 5, 3, 0, 6]_9$, \\ & & $[1,1,1,9,9,7,0,1,2,1,0,7,9,9,1,1,1]_{10}$ \\ & 13577478487477531 & $[7, 2, 8, 4, 4, 7, 6, 7, 1, 7, 6, 7, 4, 4, 8, 2, 7]_9$, \\ & & $[1,3,5,7,7,4,7,8,4,8,7,4,7,7,5,3,1]_{10}$ \\ & 14802554345520841 & $[7, 8, 8, 0, 4, 4, 4, 1, 1, 1, 4, 4, 4, 0, 8, 8, 7]_9$, \\ & & $[1,4,8,0,2,5,5,4,3,4,5,5,2,0,8,4,1]_{10}$ \\ & 54470642224607445 & $[5,4,4,7,0,6,4,2,2,2,4,6,0,7,4,4,5]_{10}$, \\ & & $[1, 2, 0, 4, 9, 0, 3, 0, 7, 0, 3, 0, 9, 4, 0, 2, 1]_{11}$ \\ & 56681764446718665 & $[5,6,6,8,1,7,6,4,4,4,6,7,1,8,6,6,5]_{10}$, \\ & & $[1, 2, 6, 2, 9, 6, 1, 4, 3, 4, 1, 6, 9, 2, 6, 2, 1]_{11}$ \\ & 56831729892713865 & $[5,6,8,3,1,7,2,9,8,9,2,7,1,3,8,6,5]_{10}$, \\ & & $[1, 2, 6, 7, 2, 3, 8, 2, 3, 2, 8, 3, 2, 7, 6, 2, 1]_{11}$ \\ & 62712119691121726 & $[6,2,7,1,2,1,1,9,6,9,1,1,2,1,7,2,6]_{10}$, \\ & & $[1, 4, 0, 1, 6, 0, 1, 7, 6, 7, 1, 0, 6, 1, 0, 4, 1]_{11}$ \\ & 64224652625642246 & $[6,4,2,2,4,6,5,2,6,2,5,6,4,2,2,4,6]_{10}$, \\ & & $[1, 4, 4, 1, 3, 10, 5, 4, 2, 4, 5, 10, 3, 1, 4, 4, 1]_{11}$ \\ \hline
19 & 6411682614162861146 & $[6,4,1,1,6,8,2,6,1,4,1,6,2,8,6,1,1,4,6]_{10}$, \\ & & $[1, 1, 7, 5, 9, 10, 6, 7, 4, 6, 4, 7, 6, 10, 9, 5, 7, 1, 1]_{11}$ \\ & 7861736017106371687 & $[7,8,6,1,7,3,6,0,1,7,1,0,6,3,7,1,6,8,7]_{10}$, \\ & & $[1, 4, 6, 1, 0, 4, 5, 4, 1, 5, 1, 4, 5, 4, 0, 1, 6, 4, 1]_{11}$ \\ \hline
21 & 104618510424015816401 & $[8, 5, 4, 0, 1, 3, 3, 4, 0, 4, 1, 4, 0, 4, 3, 3, 1, 0, 4, 5, 8]_9$, \\ & & $[1,0,4,6,1,8,5,1,0,4,2,4,0,1,5,8,1,6,4,0,1]_{10}$ \\ & 686833076121670338686 & $[6,8,6,8,3,3,0,7,6,1,2,1,6,7,0,3,3,8,6,8,6]_{10}$, \\ & & $[1, 0, 2, 5, 9, 5, 4, 1, 7, 7, 4, 7, 7, 1, 4, 5, 9, 5, 2, 0, 1]_{11}$ \\ & 771341832818238143177 & $[7,7,1,3,4,1,8,3,2,8,1,8,2,3,8,1,4,3,1,7,7]_{10}$ \\ & & $[1, 1, 6, 8, 0, 7, 1, 1, 3, 10, 1, 10, 3, 1, 1, 7, 0, 8, 6, 1, 1]_{11}$ \\ & 903253059636950352309 & $[9,0,3,2,5,3,0,5,9,6,3,6,9,5,0,3,5,2,3,0,9]_{10}$ \\ & & $[1, 3, 8, 5, 0, 4, 6, 6, 7, 10, 9, 10, 7, 6, 6, 4, 0, 5, 8, 3, 1]_{11}$ \\ \hline
23 & 89403957605050675930498 & $[8,9,4,0,3,9,5,7,6,0,5,0,5,0,6,7,5,9,3,0,4,9,8]_{10}$, \\ & & $[1, 1, 0, 9, 9, 0, 10, 6, 6, 10, 6, 2, 6, 10, 6, 6, 10, 0, 9, 9, 0, 1, 1]_{11}$ \\ \hline
25 & 9986831781362631871386899 & $[9,9,8,6,8,3,1,7,8,1,3,6,2,6,3,1,8,7,1,3,8,6,8,9,9]_{10}$, \\ & & $[1, 0, 1, 7, 5, 8, 7, 5, 2, 10, 9, 3, 3, 3, 9, 10, 2, 5, 7, 8, 5, 7, 1, 0, 1]_{11}$ \\
\end{supertabular} \end{center}

\section{Future Directions}

There are a few papers in the literature which focus on palindromes in different base systems.  For example, \cite{MR1126663} considers those palindromes which are perfect squares.  \cite{MR2321593} generalizes the question by considering those which are perfect powers.  \cite{MR740761} presents some results on the number of ways an integer can be expressed as a palindrome in different bases.  In fact, we present the following problem:

\begin{quote} \textit{What is the largest list of bases $b$ for which an integer $N \geq 10$ is a $d$-digit palindrome base $b$ for every base in the list?}  \end{quote}

\noindent If one chooses $N = 66, \, 88, \, 676, \, 989$, it is easy to see that there exists a $d$-digit palindrome base 10 has at least four different bases $b$ for which it is a $d$-digit palindrome base $b$.  It is unclear whether this is an upper bound on the number of different bases.

\end{document}